\documentclass{ws-ijnt}
\newcommand{\N}{\mathbb{N}}
\newcommand{\R}{\mathbb{R}}
\newcommand{\Z}{\mathbb{Z}}
\newcommand{\C}{\mathbb{C}}
\newcommand{\Q}{\mathbb{Q}}

\newcommand{\Ha}{\mathbb{H}}

\begin{document}

\markboth{R. Moy}
{Congruences Among Power Series Coefficients of Modular Forms}

%
\catchline{}{}{}{}{}
%

\title{CONGRUENCES AMONG POWER SERIES COEFFICIENTS OF MODULAR FORMS
}

\author{RICHARD MOY}

\address{Department of Mathematics, Northwestern University\\
Evanston, IL 60208, USA\\
\email{ramoy88@math.northwestern.edu}}

\maketitle

\begin{abstract}
Many authors have investigated the congruence relations amongst the coefficients of power series expansions of modular forms $f$ in modular functions $t$. In a recent paper, R. Osburn and B. Sahu examine several power series expansions and prove that the coefficients exhibit congruence relations similar to the congruences satisfied by the Ap\'ery numbers associated with the irrationality of $\zeta(3)$. We show that many of the examples of Osburn and Sahu are members of infinite families.
\end{abstract}

\keywords{Coefficients of power series; congruences; modular forms.}

\ccode{Mathematics Subject Classification 2000: 11F33, 11F11}

\section{Introduction}
The authors in \cite{CCS,CKKO,JV,OS1,OS2,V} investigate congruences amongst the power series coefficients of a modular form $f$ in a modular function $t$. These congruence relations are similar to those satisfied by the Ap\'ery numbers associated with the irrationality of $\zeta(3)$. In \cite{PS}, Peters and Stienstra show that the Ap\'ery numbers
\[
\alpha_n\mathrel{\mathop:}=\sum_{j=0}^n {\binom{n}{j}^2 \binom{n+j}{j} ^2}
\]
arise as coefficients of the power series expansion of
\[
f(z)=\frac{\eta^7(2z)\eta^7(3z)}{\eta^5(z)\eta^5(6z)}
\]
in
\[
t(z)=\left(\frac{\eta(z)\eta(6z)}{\eta(2z)\eta(3z)}\right)^{12},
\]
where
\[
\eta(z)=q^{1/24}\prod_{n\geq 1} {(1-q^n)},\ (q\mathrel{\mathop:}=e^{2\pi i  z})
\]
is Dedekind's eta function. The following two examples are discussed by Osburn and Sahu in \cite{OS1}.

\subsection{Example 1}
Let
\[
f=\frac{\eta^5(z)}{\eta(5z)}=1 - 5q + 5q^2 + 10q^3 - 15q^4 - 5q^5 - 10q^6 + 30q^7 + 25q^8 + \dots
\]
and let
\[
t=\frac{\eta^6(5z)}{\eta^6(z)}=q + 6q^2 + 27q^3 + 98q^4 + 315q^5 + 912q^6 + 2456q^7 + 6210q^8 \dots.
\]
Write
\[
f=\sum_{n=0}^\infty {b_nt^n}=1 -5t+ 35t^2 -275t^3+ 2275t^4 -19255t^5+ 163925t^6 -1385725t^7+ 11483875t^8+ \dots.
\]
For all primes $p$, we have $b_{\ell p^r}\equiv b_{\ell p^{r-1}}\mod p^r$ for all $\ell,r\in\N$.
\subsection{Example 2}
Let
\[
f=\frac{\eta^4(z)\eta^4(3z)}{\eta^2(2z)\eta^2(6z)}=1 - 4q + 4q^2 - 4q^3 + 20q^4 - 24q^5 + 4q^6 - 32q^7 + 52q^8+\dots
\]
and let
\[
t=\frac{\eta^6(2z)\eta^6(6z)}{\eta^6(z)\eta^6(3z)}=q + 6q^2 + 21q^3 + 68q^4 + 198q^5 + 510q^6 + 1248q^7 + 2904q^8 +\dots.
\]
Write
\[
f=\sum_{n=0}^\infty {\tilde{b}_nt^n}=1 -4t +28t^2 -256t^3+ 2716t^4 -31504t^5+ 387136t^6 -4951552t^7+ 65218204t^8 +\dots.
\]
We have $\tilde{b}_n=(-1)^n\sum_{k=0}^n {{\binom{n}{k}}^2\binom{2k}{k}\binom{2(n-k)}{n-k}}$. For all primes $p\nmid 6$, we have \mbox{$\tilde{b}_{\ell p^r}\equiv \tilde{b}_{\ell p^{r-1}}\mod p^r$} for all $l,r\in \N$.
\newline
In this paper, we prove that the examples above and in \cite[Tables 2--3]{OS1} are all members of infinite families (which are given in Theorem \ref{theorem1} and Theorem \ref{theorem2}).

Fix an integer $N\in\N$ and a Dirichlet character $\epsilon$ modulo $N$. Let $
\Gamma_0(N)\mathrel{\mathop:}= \{ \bigl(\begin{smallmatrix} a & b \\ c & d\end{smallmatrix}\bigr)\in \textup{SL}_2(\Z)| c\equiv 0 \mod N\}
$. Let $M_0^!(\Gamma_0(N),\epsilon)$ and $\mathcal{M}_0(\Gamma_0(N),\epsilon)$ be respectively the space of weakly holomorphic modular functions and the space of meromorphic modular functions on $\Gamma_0(N)$ with character $\epsilon$. For an integer $k\in\Z$, let $M_k(\Gamma_0(N),\epsilon)$ be the space of modular forms on $\Gamma_0(N)$ with character $\epsilon$, and let $E_k(\Gamma_0(N),\epsilon)$ be its Eisenstein subspace.
Let
\[
\widetilde{M}_0^!(\Gamma_0(N),\epsilon)=M_0^!(\Gamma_0(N),\epsilon)\cap \Q\left[\left[q,q^{-1}\right]\right]
 \]
and let
\[
\widetilde{\mathcal{M}}_0(\Gamma_0(N),\epsilon)= \mathcal{M}_0(\Gamma_0(N),\epsilon)\cap \Q\left[\left[q,q^{-1}\right]\right].
\]
Similarly, let
\[
\widetilde{M}_k(\Gamma_0(N),\epsilon)=M_k(\Gamma_0(N),\epsilon)\cap \Q\left[\left[q\right]\right]
 \]
and let
\[
\widetilde{E}_k(\Gamma_0(N),\epsilon)=E_k(\Gamma_0(N),\epsilon)\cap \Q\left[\left[q\right]\right].
\]
For $k>0$, $N\in\N$, and $\epsilon$ a Dirichlet character modulo, we say that $M_k(\Gamma_0(N),\epsilon)$ satisfies \eqref{conditionstar} if
\begin{equation*}\label{conditionstar}
\dim \widetilde{M_{k+2}}(\Gamma_0(N),\epsilon)<\dim \widetilde{M_k}(\Gamma_0(N),\epsilon)+\dim \widetilde{E_k}(\Gamma_0(N),\epsilon).\tag{$\star$}
\end{equation*} For additional notation, see Section \ref{sect_prelim}.

\begin{theorem}\label{theorem1}
Choose an integer $ N$, a real-valued Dirichlet character $\chi$ modulo $N$, and a positive integer $k$ such that $\widetilde {M_k}(\Gamma_0(N),\chi)$ satisfies \eqref{conditionstar}. Let
\[
t=q+\sum_{n=2}^\infty{a_nq^n}\in \widetilde{M}_0^!(\Gamma_0(N),\psi)
\]
such that the zeros and poles of $t$ are supported at the cusps, where $\psi$ is any Dirichlet character modulo $N$. Then there exists a non-zero $f\in \widetilde{M}_k(\Gamma_0(N),\chi)$ such that when we write
\[
f=\sum_{n=0}^\infty {b_n t^n},
\]
we have
\[
b_{\ell p^r}\equiv b_{\ell p^{r-1}}\mod p^r
\]
for all but finitely many primes $p$ described in Table \ref{table_theorem1hypoth} and for all $\ell,r\in\N$.
\end{theorem}
\begin{remark}
By dimension formulas \cite[p. 85-92]{DS}, condition \eqref{conditionstar} is satisfied by all integers $N$, Dirichlet characters $\chi$ modulo $N$, and positive integers $k$ in Table \ref{table_theorem1hypoth}.
\end{remark}

\begin{corollary}\label{theorem1_cor}
Choose an integer $N=2,3,5,7,$ or $13$, a real-valued Dirichlet character $\chi$ modulo $N$, and a positive integer $k$ such that
\[
\dim \widetilde{M_k}(\Gamma_0(N),\chi)=\dim \widetilde{M_{k+2}}(\Gamma_0(N),\chi).
\]
Let
\[
t=q+\sum_{n=2}^\infty{a_nq^n}\in \widetilde{M}_0^!(\Gamma_0(N),\psi)
\]
such that the zeros and poles of $t$ are supported at the cusps, where $\psi$ is any Dirichlet character modulo $N$. Then there exist non-zero $f_1,f_2\in \widetilde{M}_k(\Gamma_0(N),\chi)$ such that when we write
\[
f_1=\sum_{n=0}^\infty {b_{1,n} t^n},
\]
and
\[
f_2=\sum_{n=0}^\infty {b_{2,n} t^n},
\]
we have
\[
b_{1,\ell p^r}\equiv b_{1,\ell p^{r-1}}\mod p^r
\]
and
\[
b_{2,\ell p^r}\equiv \chi(p) b_{2,\ell p^{r-1}}\mod p^r
\]
for all but finitely many primes $p$ and for all $\ell,r\in\N$.
\end{corollary}

\begin{remark}
The conditions of Corollary \ref{theorem1_cor} are satisfied by the integers $N$, Dirichlet characters $\chi$ modulo $N$, and positive integers $k$ found in Table \ref{table_theorem2hypoth}.
\end{remark}

We obtain Example 1 for primes $p$ satisfying $\left(\frac{p}{5}\right)=1$ from Theorem \ref{theorem1} by choosing $t=\frac{\eta^6(5z)}{\eta^6(z)}$ and $f=\frac{\eta^5(z)}{\eta(5z)}$. We show how to obtain Example 1 for all primes $p$ in Section \ref{example1_rev}. We obtain Example 2 from Theorem \ref{theorem1} by choosing $t=\frac{\eta^6(2z)\eta^6(6z)}{\eta^6(z)\eta^6(3z)}$ and $f=\frac{\eta^4(z)\eta^4(3z)}{\eta^2(2z)\eta^2(6z)}$. See Section \ref{sect_examples1} for additional examples. See Table \ref{table_weight0etaproducts} for more examples of modular functions $t$ satisfying the hypotheses of Theorem \ref{theorem1}.

\begin{table}[!htbp]
\parbox{.5\linewidth}{
\centering
\caption{Levels, characters, weights, and primes for Theorem \ref{theorem1}}
\begin{tabular}{|c c l c|}
\hline
$N$& $\chi$ & Values of $k$ & Values of $p$ \\ \hline
2 & $\chi_0$&$0\mod 2$& $p\ne 2$
\\
3 & $\chi_0$&$0\mod 2$& $p\ne 3$
\\
 & $\left(\frac{\bullet}{3}\right)$& $1\mod 2$ &$\left(\frac{p}{3}\right)=1$
\\
4 & $\chi_0$&$0\mod 2$&$p\ne 2$
\\
& $\left(\frac{-4}{\bullet}\right)$& $1\mod 2$&$\left(\frac{-4}{p}\right)=1$
\\
5 &$\chi_0$& $0\mod 4$&$p\ne 5$
\\
&$\left(\frac{\bullet}{5}\right)$ & $2\mod 4$&$\left(\frac{p}{5}\right)=1$
\\
6 & $\chi_0$&$0\mod 2$&$p\ne 2,3$
\\
& $\left(\frac{\bullet}{3}\right)$& $1\mod 2$&$\left(\frac{p}{3}\right)=1$
\\
7 &$\chi_0$& $0\mod 6$&$p\ne 7$
\\
&$\left(\frac{\bullet}{7}\right)$ & $3\mod 6$&$\left(\frac{p}{7}\right)=1$
\\
8 & $\chi_0$&$0\mod 2$&$p\ne 2$
\\
& $\left(\frac{-4}{\bullet}\right)$& $1\mod 2$&$\left(\frac{-4}{p}\right)=1$
\\
9 & $\chi_0$&$0\mod 2$&$\left(\frac{p}{3}\right)=1$
\\
 & $\left(\frac{\bullet}{3}\right)$& $1\mod 2$&$\left(\frac{p}{3}\right)=1$
\\
10&$\chi_0$&$0\mod 4$&$p\ne 2,5$
\\
&$ \left(\frac{\bullet}{5}\right)$ & $0\mod 2$&$\left(\frac{p}{5}\right)=1$
\\
12 & $\chi_0$& $0\mod 2$&$p\ne 2,3$
\\
& $\left(\frac{\bullet}{3}\right)$& $1\mod 2$&$\left(\frac{p}{3}\right)=1$
\\
13 & $\chi_0$&$0\mod 12$&$p\ne 13$
\\
&$\left(\frac{\bullet}{13}\right)$ & $6\mod 12$&$\left(\frac{p}{13}\right)=1$
\\
16 &$\chi_0$&  $0\mod 2$&$\left(\frac{-4}{p}\right)=1$
\\
& $\left(\frac{-4}{\bullet}\right)$& $1\mod 2$&$\left(\frac{-4}{p}\right)=1$
\\
18 &$\chi_0$&  $0\mod 2$&$p\equiv 1\mod 6$
\\
& $ \left(\frac{\bullet}{3}\right)$& $1\mod 2$&$\left(\frac{p}{3}\right)=1$
\\
25 & $\chi_0$& $0\mod 4$&$p\equiv 1\mod 5$
\\
\hline
\end{tabular}
\label{table_theorem1hypoth}
}
\hfill
\parbox{.4\linewidth}{
\centering
\caption{Examples of weakly holomorphic modular forms of weight 0 for Theorem \ref{theorem1}}
\begin{tabular}{|c c|}
\hline
Subgroup & $t\in \widetilde{M}_0^!(\Gamma_0(N),\psi)$\\ \hline
$\Gamma_0(2)$&$\frac{\eta^{24}(2z)}{\eta^{24}(z)}$
\\
$\Gamma_0(3)$& $\frac{\eta^{12}(3z)}{\eta^{12}(z)}$
\\
$\Gamma_0(4)$& $\frac{\eta^8(z)\eta^{16}(4z)}{\eta^{24}(2z)}$
\\
$\Gamma_0(5)$& $\frac{\eta^6(5z)}{\eta^6(z)}$
\\
$\Gamma_0(6)$&$\left(\frac{\eta(2z)\eta(6z)}{\eta(z)\eta(3z)}\right)^6$
\\
$\Gamma_0(7)$& $\frac{\eta^4(7z)}{\eta^4(z)}$
\\
$\Gamma_0(8)$ & $\frac{\eta^2(2z)\eta^4(8z)}{\eta^4(z) \eta^2(4z)}$
\\
$\Gamma_0(9)$& $\frac{\eta^3(9z)}{\eta^3(z)}$
\\
$\Gamma_0(10)$ & $\left(\frac{\eta(5z)\eta(10z)}{\eta(z)\eta(2z)}\right)^2$
\\
$\Gamma_0(12)$& $\left(\frac{\eta(4z)\eta(12z)}{\eta(z)\eta(3z)} \right)^2$
\\
$\Gamma_0(13)$&$\frac{\eta^2(13z)}{\eta^2(z)}$
\\
$\Gamma_0(16)$& $\frac{\eta(2z)\eta^2(16z)}{\eta^2(z)\eta(8z)}$
\\
$\Gamma_0(18)$ & $\frac{\eta(2z)\eta(3z)\eta^2(18z)}{\eta^2(z)\eta(6z)\eta(9z)}$
\\
$\Gamma_0(25)$& $\frac{\eta(25z)}{\eta(z)}$
\\
\hline
\end{tabular}
\label{table_weight0etaproducts}
}
\end{table}

\begin{remark}
When $\Gamma_0(N)$ is genus zero and $N\ne 16,18,25$, all the weakly holomorphic modular functions $t$ that satisfy the conditions in Theorem \ref{theorem1} are products of the Dedekind eta functions $\eta(dz)$, $d|N$. To see this, we observe that $\frac{qdt}{tdq}$ is a holomorphic modular form of weight 2 on $\Gamma_0(N)$. If $N\ne 16,18,25$, then the space of modular forms of weight 2 on $\Gamma_0(N)$ is spanned by $E_2(z)-dE_2(dz)$, $d|N$. This implies that $t$ is a product of $\eta(dz)$ for $d|N$.
\end{remark}

Theorem \ref{theorem1} produces many infinite families of modular forms $f$ and modular functions $t$ such that the power series coefficients of $f$ in $t$ exhibit congruence relations. However, the zeros and poles of $t$ must be supported at the cusps. Theorem \ref{theorem2} will produce infinite families of such modular forms $f$ and modular functions $t$ without this restriction on $t$.

\begin{theorem}\label{theorem2}
From Table \ref{table_theorem2hypoth}, choose an integer $N$, a Dirichlet character $\chi$ modulo $N$, and a positive integer $k$. Choose any
\[
g=q+\sum_{n=2}^\infty{g_nq^n}\in \widetilde{M}_k(\Gamma_0(N),\chi)
\]
such that the zeros of $g$ are supported at the cusps. There exist
\[
f=1+\sum_{n=1}^\infty {f_nq^n} \in \widetilde{M}_k(\Gamma_0(N),\chi)
\]
and $t\in \widetilde{\mathcal{M}}_0(\Gamma_0(N))$ such that when we write
\[
f=\sum_{n=0}^\infty {b_n t^n},
\]
we have $g=f\cdot t$ and
\[
b_{\ell p^r}\equiv b_{\ell p^{r-1}}\mod p^r
\]
for all but finitely many primes $p$ satisfying $\chi(p)=1$ and for all $\ell,r\in\N$ .
\end{theorem}

We obtain Example 1 for primes $p$ satisfying $\left(\frac{p}{5}\right)=1$ from Theorem \ref{theorem2} by choosing $g=\frac{\eta^5(5z)}{\eta(z)}$, choosing $f=\frac{\eta^5(z)}{\eta(5z)}$, and choosing $t=\frac{\eta^6(5z)}{\eta^6(z)}$. See Section \ref{sect_examples2} for additional examples. For examples of $g$ satisfying the hypotheses of Theorem \ref{theorem2}, choose $A$ and $B$ from Table \ref{table_etaproducts} and an integer $n$ such that $g=A\cdot B^n\in \widetilde{M}_k(\Gamma_0(N),\chi).$

\begin{table}[!htbp]
\parbox{.3\linewidth}{
\centering
\caption{Levels, characters, and weights for Theorem \ref{theorem2}}
\begin{tabular}{|c c l|}
\hline
N &$\chi$ &{Values of k}\\
\hline
2 &$\chi_0$&$0\mod 4$\\
3 & $\chi_0$& $0\mod 6$\\
& $ \left(\frac{\bullet}{3}\right) $&$3\mod 6$\\
5&$\chi_0$ & $0\mod 4$\\
& $ \left(\frac{\bullet}{5}\right) $ & $2\mod 4$\\
7 & $\chi_0$&$0\mod 6$\\
& $\left(\frac{\bullet}{7}\right)$&$3\mod 6$\\
13 & $\chi_0$&$0\mod 12$\\
& $\left(\frac{\bullet}{13}\right)$&$6\mod 12$\\
\hline
\end{tabular}
\label{table_theorem2hypoth}
}
\hfill
\parbox{.5\linewidth}{
\centering
\caption{Examples of modular forms \mbox{$g=AB^n$} for Theorem \ref{theorem2}}
\begin{tabular}{|c c c c|}
\hline
Subgroup & $\chi$& $A$ & $B$\\ \hline
$\Gamma_0(2)$ & $\chi_0$&$\frac{\eta(2z)^{16}}{\eta(z)^8}$& $\frac{\eta(z)^{16}}{\eta(2z)^8}$\\
$\Gamma_0(3)$&$\chi_0$&$\eta(3z)^6 \eta(z)^6$ &$\frac{\eta(z)^{18}}{\eta(3z)^6}$ \\
& $ \left(\frac{\bullet}{3}\right)$ & $ \frac{\eta(3z)^9}{\eta(z)^3}$ &$\frac{\eta(z)^{18}}{\eta(3z)^6}$\\
$\Gamma_0(5)$& $\chi_0$&$\eta(5z)^4\eta(z)^4$ & $ \frac{\eta(z)^{10}}{\eta(5z)^2}$\\
&$ \left(\frac{\bullet}{5}\right) $&$\frac{\eta(5z)^5}{\eta(z)}$&$ \frac{\eta(z)^{10}}{\eta(5z)^2}$\\
$\Gamma_0(7)$& $\chi_0$&$\eta(z)^{10}\eta(7z)^2$&$\frac{\eta(z)^{14}}{\eta(7z)^2}$ \\
&$\left(\frac{\bullet}{7}\right)$&$\eta(z)^3\eta(7z)^3$&$\frac{\eta(z)^{14}}{\eta(7z)^2}$\\
$\Gamma_0(13)$ &$\chi_0$& $\eta(z)^{24}$& $\frac{\eta(z)^{26}}{\eta(13z)^2}$  \\
& $\left(\frac{\bullet}{13}\right)$&$ \eta(z)^{11}\eta(13z)$& $\frac{\eta(z)^{26}}{\eta(13z)^2}$ \\
\hline
\end{tabular}
\label{table_etaproducts}
}
\end{table}

We begin with a few preliminaries in \mbox{Section \ref{sect_prelim}}, and we prove the primary tool for obtaining congruences in \mbox{Section \ref{sect_congrtool}}. In \mbox{Sections \ref{sect_proof1} and \ref{sect_proof2}}, we prove \mbox{Theorem \ref{theorem1}} and \mbox{Theorem \ref{theorem2}}. We give examples from these theorems in \mbox{Sections \ref{sect_examples1} and \ref{sect_examples2}} respectively. In Section \ref{sect_additional_examples}, we provide examples of other families of congruences obtained via the method of Theorem \ref{theorem1}. These additional examples were pointed out by the referee, whom we would like to especially thank. There are many other examples of families of modular forms and modular functions exhibiting similar properties; however, no single paper can be exhaustive. We have included the examples in this paper to show that congruences of coefficients of power series expansions of modular forms in modular functions are not uncommon.

\section{Preliminaries}\label{sect_prelim}
Let $\Ha$ denote the upper half-plane. If $f(z)$ is a meromorphic function with period 1 on $\Ha$ with Fourier expansion
\[
f(z)=\sum_{n=h}^\infty {f_n q^n},
\]
 then define
\[
\theta f(z)\mathrel{\mathop:}= \frac{1}{2\pi i} \frac{d}{dz} f(z)=\sum_{n=h}^\infty {n f_n q^n}.
\]
For a modular form or modular function $f$, let $\nu_\infty(f)$ be the order of vanishing of $f$ at $\infty$. Let $\left(\frac{\bullet}{p}\right)$ denote the Jacobi symbol and $\chi_0$ denote the principal Dirichlet character (whose modulus will be clear from context).

We can write an explicit basis for $E_{k}(\Gamma_0(N),\epsilon)$ where $\epsilon$ is a Dirichlet character modulo $N$. For further details, see \cite[Chapter 4]{DS}. Suppose $\chi,\psi$ are primitive characters with conductors $L$ and $R$ respectively. Let
\[
E_{k,\chi,\psi}(z)\mathrel{\mathop:}= c_0+\sum_{n=1}^\infty {\left( \sum_{m|n}{\psi(m)\chi(n/m)m^{k-1}}\right)q^n}
\]
where
\[ c_0 = \left\{ \begin{array}{ll}
 -\frac{B_{k,\psi}}{2k} & \mbox{if $L=1$},\\
0 & \mbox{if $L>1$}
\end{array} \right. \]
and $B_{k,\psi}$ is the generalized Bernoulli number associated with $\psi$. If $d$ is a positive integer and $k\geq 3$ is an integer such that $\chi(-1)\psi(-1)=(-1)^k$, then \mbox{$E_{k,\chi,\psi}(dz)\in M_{k}(\Gamma_0(RLd),\chi\psi)$.} Moreover, given $N$ and $\epsilon$, the series $E_{k,\chi,\psi}(dz)$ with $RLd|N$ and $\chi\psi=\epsilon$ form a basis for $E_{k}(\Gamma_0(N),\epsilon)$.

Suppose $\chi,\psi$ are primitive characters with conductors $L$ and $R$ respectively. Further suppose $d,k\in\N$ and $RLd|N$. Let $E_{k,\chi,\psi}(dz)=\sum{e_nq^n}$, let $p$ be any prime not dividing $d$, and let $\ell\in\N$. If $d|\ell$ then
\begin{equation}\label{eisenstein_eq1}
e_{\ell p^r}-\chi(p)e_{\ell p^{r-1}}=\sum_{m|\frac{\ell}{d}} {\chi\left(\tfrac{\ell}{md}\right)\psi(mp^r)m^{k-1}p^{r(k-1)}}\equiv 0\mod p^r.
\end{equation}
and if $d\nmid \ell$ then also
\begin{equation}\label{eisenstein_eq2}
e_{\ell p^r}-\chi(p)e_{\ell p^{r-1}}=0\equiv 0\mod p^r.
\end{equation}

\begin{remark}
From Table \ref{table_theorem1hypoth}, choose an integer $N$, a Dirichlet character $\epsilon$ modulo $N$, and a positive integer $k$. The $\C$-basis for $E_k(\Gamma_0(N),\epsilon)$ described above is a $\Q$-basis for $\widetilde{E}_k(\Gamma_0(N),\epsilon)$. The dimension of $M_k(\Gamma_0(N),\epsilon)$ is also the same as the dimension of $\widetilde{M}_k(\Gamma_0(N),\epsilon)$.
\end{remark}

\begin{lemma}\label{lemma_eisensteincongr}
Suppose $\chi_i,\psi_i$ are primitive characters with conductors $L_i$ and $R_i$ respectively. Further suppose $d_i,k,N\in\N$ and $R_iL_id_i|N$. Let $\{ E_{k+2,\chi_i,\psi_i}(d_i z)\}$ be the basis for $E_{k+2}(\Gamma_0(N),\chi)$ described above. Suppose that
\begin{equation*}
\sum_{n=0}^\infty {c_nq^n}= \sum_{i} {\frac{\alpha_i}{\beta_i}E_{k+2,\chi_i,\psi_i}(d_i z)}
\end{equation*}
where $\alpha_i,\beta_i\in\Z$ and $\gcd(\alpha_i,\beta_i)=1$ for all $i$. Then for any prime $p\nmid \prod_i {\beta_id_i}$ satisfying $\chi_i(p)=1$ for all $i$, we have
\[
c_{\ell p^r}\equiv c_{\ell p^{r-1}}\mod p^r
\]
for all $\ell,r\in\N$.
\end{lemma}
\begin{proof}
Let $E_{k+2,\chi_i,\psi_i}(d_i z)=\sum{e_nq^n}$. From \eqref{eisenstein_eq1} and \eqref{eisenstein_eq2}, we have $e_{\ell p^r}\equiv e_{\ell p^{r-1}}\mod p^r$ for all $\ell,r\in\N$ and for primes $p\nmid d_i$ such that $\chi_i(p)=1$. Therefore, we have
\[
c_{\ell p^r}\equiv c_{\ell p^{r-1}}\mod p^r.
\]
for all primes $p\nmid \prod_i {\beta_id_i}$ satisfying $\chi_i(p)=1$ for all $i$ and for all $\ell,r\in\N$.
\end{proof}
\section{Tool For Obtaining Congruences}\label{sect_congrtool}
Following methods of \cite{B,JV,V}, we prove a result which allows us to extend certain congruences between related differential forms. Our result was proved by Jarvis and Verrill in the case where the sum begins with $n=1$ in \eqref{JV_Lemma_Ex_Eq1}. In \cite{OS1} it is mentioned that Stienstra has indicated that the result follows in the case where $n=0$ by formal group theory. We give a proof here for completeness. Define $\Z_{(p)}$ be the localization of $\Z$ at the prime $p$. For $x\in\Q$, define $\omega_p(x)$ to be the $p$-adic valuation of $x$.

\begin{proposition}\label{JV_Lemma_Extended}
Let $p$ be any prime. Let
\begin{equation}\label{t_power_series}
t= \sum_{n=1}^\infty {a_n u^{n}}\in\Z_{(p)}\left[\left[u\right]\right]
\end{equation}
be convergent in a neighborhood of $u=0$ in $\R$ and let $a_1$ be a unit in $\Z_{(p)}$. Suppose that in some neighborhood of $u=0$ in $\R$ we have an equality of differential forms given by
\begin{equation}\label{JV_Lemma_Ex_Eq1}
\sum_{n=0}^\infty{b_n t^{n-1}dt}=\sum_{n=0}^\infty {c_n u^{n-1}du}
\end{equation}
with $b_n,c_n\in\Z_{(p)}$ for all $n$. Then we have
\begin{equation}\label{JV_Lemma_Ex_Eq2}
b_{\ell p^r} \equiv b_{\ell p^{r-1}} (mod\ p^r)
\end{equation}
for all $\ell,r\in\N$ if and only if
\begin{equation}\label{JV_Lemma_Ex_Eq3}
c_{\ell p^r} \equiv c_{\ell p^{r-1}} (mod\ p^r)
\end{equation}
for all $\ell,r\in\N$.
\end{proposition}

We start with a short lemma whose proof follows \cite[Proposition 3.1]{V}.
\begin{lemma}\label{Power_Series_Lemma}
Let $p$ be a prime. Given a sequence $\{\alpha_n\}_{n=0}$ with $\alpha_n\in\Z_{(p)}$ and a differential form $\Omega(x)=\sum_{n=0}^\infty {\alpha_nx^{n-1}}dx$, the relation
\begin{equation*}
\alpha_n-\alpha_{n/p}\equiv 0 \mod p^{\omega_p(n)},\ \ \  n\in\Z
\end{equation*}
is equivalent to the relation
\begin{equation*}
\Omega(x)-\frac{1}{p}\Omega(x^p)=d\Psi_1(x)
\end{equation*}
for some $\Psi_1(x)\in\Z_{(p)}\left[\left[x\right]\right]$.
\end{lemma}

\begin{proof}
Since
\begin{equation*}
\Omega(x^p)=\sum_{n=0}^\infty{\alpha_n x^{p(n-1)}}d(x^{p})=p\sum_{n=0}^\infty{\alpha_n x^{pn-1}}dx=p\sum_{n=0}^\infty {\alpha_{n/p}x^{n-1}}dx,
\end{equation*}
we have
\begin{equation}\label{Power_Series_Eq}
\Omega(x)-\frac{1}{p}\Omega(x^p)=\sum_{n=0}^\infty {(\alpha_n-\alpha_{n/p})x^{n-1}}dx=\sum_{n=1}^\infty {(\alpha_n-\alpha_{n/p})x^{n-1}}dx.
\end{equation}
The lemma follows since the coefficients of the differential form in \eqref{Power_Series_Eq} are divisible by $p^{\omega_p(n)}$.
\end{proof}

\begin{proof}[Proof of Proposition \ref{JV_Lemma_Extended}]
Set
\begin{equation}\label{JV_Lemma_Ex_Eq4}
\Omega(t)=\sum_{n=0}^\infty{b_nt^{n-1}dt}.
\end{equation}
Suppose that \eqref{JV_Lemma_Ex_Eq2} holds true. By Lemma \ref{Power_Series_Lemma} we observe that the congruences in \eqref{JV_Lemma_Ex_Eq2} are equivalent to
\begin{equation}\label{JV_Lemma_Ex_Eq5}
\Omega(t)-\frac{1}{p}\Omega(t^p)=d\Psi_1(t),\ \ \ \ \Psi_1(t)\in\Z_{(p)}\left[\left[t\right]\right].
\end{equation}
By \eqref{JV_Lemma_Ex_Eq1} we have
\begin{equation*}
\Omega(t(u))=\sum_{n=0}^\infty {c_n u^{n-1}du}.
\end{equation*}
Once we show that
\begin{equation}\label{desired_eq1}
\Omega(t(u))-\Omega(t(u^p))=d\Psi(u)
\end{equation}
for some $\Psi(u)\in\Z_{(p)}\left[\left[u\right]\right]$ we may conclude by Lemma \ref{Power_Series_Lemma} that \eqref{JV_Lemma_Ex_Eq3} holds for all $\ell,r\in\N$. Since $a_1$ is a unit in $\Z_{(p)}$, we can invert the power series in \eqref{t_power_series} and use the same proof to show that \eqref{JV_Lemma_Ex_Eq3} implies \eqref{JV_Lemma_Ex_Eq2}. Therefore, we only need to prove that \eqref{JV_Lemma_Ex_Eq2} implies \eqref{JV_Lemma_Ex_Eq3}.

To prove \eqref{desired_eq1} we first observe
\begin{equation}\label{JV_Lemma_Ex_Eq6}
dt(u)=\sum_{n=1}^\infty {a_n n u^{n-1}du}
\end{equation}
and
\begin{equation}\label{JV_Lemma_Ex_Eq7}
dt(u^p)=\sum_{n=1}^\infty {a_n np u^{pn-1}du}.
\end{equation}
Also observe that
\begin{equation}\label{JV_Lemma_Ex_Eq8}
t(u)^{np}=t(u^p)^n+np\xi_n(u),
\end{equation}
where
\[
\xi_n(u)=\sum_{i=np}^\infty {e_iu^i}\in\Z_{(p)}\left[\left[u\right]\right].
\]
We see that $\sum_{n} {\xi_n(u)}$ is a well defined power series in $u$. Using \eqref{JV_Lemma_Ex_Eq4}, \eqref{JV_Lemma_Ex_Eq6}, \eqref{JV_Lemma_Ex_Eq7}, \eqref{JV_Lemma_Ex_Eq8} and noting that the $n=0$ terms cancel, we see that
\begin{align}\label{JV_Lemma_Ex_Eq9}
\notag d\Psi_1(t(u)) & =\Omega(t(u))-\frac{1}{p}\Omega(t(u)^p)\\ \notag
&=\sum_{n=1}^\infty {\frac{b_n}{n}{d(t(u)^n)}}-\sum_{n=1}^\infty {\frac{b_n}{np}d(t(u)^{pn})}
\\ \notag
&=\sum_{n=1}^\infty {\frac{b_n}{n}{d(t(u)^n)}}-\sum_{n=1}^\infty {b_n\cdot d\left(\frac{t(u^p)^n}{pn}+\xi_n(u) \right)}
\\
&=\Omega(t(u))-\frac{1}{p}\Omega(t(u^p))-
b_0\frac{dt(u)}{t(u)}+\frac{b_0}{p}\frac{dt(u^p)}{t(u^p)}-d\Psi_2(u)
\end{align}
for some $\Psi_2(u)\in\Z_{(p)}\left[\left[u\right]\right]$. Thus to prove \eqref{desired_eq1} we need only show that
\begin{equation}\label{dt_final_eq}
\frac{dt(u)}{t(u)}-\frac{dt(u^p)}{pt(u^p)}=d\Psi_3(u)
\end{equation}
for some $\Psi_3(u)\in\Z_{(p)}\left[\left[u\right]\right]$.

Since $t(u)$ and $t(u^p)$ are convergent in a neighborhood of $u=0$ in $\R$ and $\frac{t(u)^p}{t(u^p)}\neq 0$ in a neighborhood of $u=0$, we can take the $\log$ of $\frac{t(u)^p}{t(u^p)}$. Observe that
\begin{align*}
\frac{1}{p}\log\left(\frac{t(u)^p}{t(u^p)} \right)=\log\left(1+\tilde{t}(u)\right)-\frac{1}{p}\log(1+\tilde{t}(u^p))
\end{align*}
where $\tilde{t}(u)=\sum_{n=2}^\infty{\frac{a_n}{a_1}u^{n-1}}\in\Z_{(p)}\left[\left[ u\right]\right]$. Therefore,
\begin{align*}
\frac{1}{p}\log\left(\frac{t(u)^p}{t(u^p)} \right)=\sum_{n=1}^\infty{(-1)^{n+1}\frac{\tilde{t}(u)^n}{n}}-\frac{1}{p}\sum_{n=1}^\infty{(-1)^{n+1}\frac{\tilde{t}(u^p)^n}{n}}.
\end{align*}
Differentiating gives
\begin{align*}
\frac{dt(u)}{t(u)}-\frac{dt(u^p)}{pt(u^p)}&=\sum_{n=1}^\infty{(-1)^{n+1}\left(\frac{d(\tilde{t}(u)^n)}{n}-\frac{d(\tilde{t}(u^p)^n)}{np}\right)},
\end{align*}
which is a differential form with coefficients in $\Z_{(p)}$. By the same reasoning as in \eqref{JV_Lemma_Ex_Eq8} we have
\[
\frac{dt(u)}{t(u)}-\frac{dt(u^p)}{pt(u^p)}=\sum_{n=1}^\infty {(-1)^{n+1}\left(\frac{d(\tilde{t}(u)^n)}{n}-\frac{d(\tilde{t}(u)^{np}-np\tilde{\xi}_n(u))}{np} \right)}
\]
for some
\[
\tilde{\xi}_n(u)=\sum_{i=np}^\infty {\tilde{e}_i u^i}\in\Z_{(p)}\left[\left[u\right]\right].
\]
Hence,
\begin{align}\label{JV_Lemma_Ex_Eq10}
\notag \frac{dt(u)}{t(u)}-\frac{dt(u^p)}{pt(u^p)}&=\sum_{n=1}^\infty {(-1)^{n+1}\left(\frac{d(\tilde{t}(u)^n)}{n}-\frac{d(\tilde{t}(u)^{np})}{np}+d(\tilde{\xi}_n(u))\right)}
\\
&=\tilde{\Omega}(\tilde{t}(u))-\frac{1}{p}\tilde{\Omega}(\tilde{t}(u)^p)+d\Psi_4(u)
\end{align}
for some $\Psi_4(u)\in\Z_{(p)}\left[\left[u\right]\right]$ where
\[ \tilde{\Omega}(\tilde{t})=\sum_{n=1}^\infty{(-1)^{n+1}\tilde{t}^{n-1}d\tilde{t}}=\sum_{n=1}^\infty{d_n\tilde{t}^{n-1}d\tilde{t}}.
\]
Since $d_{\ell p^r}\equiv d_{\ell p^{r-1}}\mod p^r$ for all primes $p$ and for all $\ell,r\in\N$, we have
\begin{equation}\label{JV_Lemma_Ex_Eq11}
\tilde{\Omega}(\tilde{t(u)})-\frac{1}{p}\tilde{\Omega}(\tilde{t}^p(u))=d\Psi_5(\tilde{t}(u))
\end{equation}
for some $\Psi_5(\tilde{t})\in\Z_{(p)}\left[\left[\ \tilde{t}\ \right]\right]$ by Lemma \ref{Power_Series_Lemma}. Therefore, \eqref{dt_final_eq} holds by combining \eqref{JV_Lemma_Ex_Eq10} and \eqref{JV_Lemma_Ex_Eq11}.
\end{proof}

\section{Proof of Theorem \ref{theorem1}}\label{sect_proof1}
Choose an integer $N$, a Dirichlet character $\chi$ modulo $N$, and a positive integer $k$ such that $\widetilde{M_k}(\Gamma_0(N),\chi)$ satisfies \eqref{conditionstar}. Fix
\[
t=q+\sum_{n=2}^\infty{a_nq^n}\in \widetilde{M}_0(\Gamma_0(N),\psi),
\]
where $\psi$ is any Dirichlet character modulo $N$, such that the zeros and poles of $t$ are supported at the cusps (See Table \ref{table_weight0etaproducts} for examples of such $t$). We use multiplication by $\frac{q\frac{dt}{dq}}{t}$ as a linear transformation to find $f\in \widetilde{M}_k(\Gamma_0(N),\chi)$ such that the power series coefficients of $f$ in $t$ satisfy congruence relations.
\begin{remark}
Because the zeros and poles of $t$ are supported at the cusps, we have $\frac{q\frac{dt}{dq}}{t}\in \widetilde{M}_2(\Gamma_0(N))$. Therefore, multiplication by $\frac{q\frac{dt}{dq}}{t}$ is an injective linear transformation$\colon$
\begin{align*}
\notag \widetilde{M}_k(\Gamma_0(N),\chi) &\longmapsto  \widetilde{M}_{k+2}(\Gamma_0(N),\chi)
\\
f &\longmapsto  f\frac{q\frac{dt}{dq}}{t}.
\end{align*}
\end{remark}

\begin{proof}[Proof of Theorem \ref{theorem1}]
Let $N$, $\chi$, $k$, $t$, and $\psi$ be the same as above.Recall that $\widetilde{M_k}(\Gamma_0(N),\chi)$ satisfies \eqref{conditionstar}, which implies that
\begin{equation*}
\dim \widetilde{M}_k(\Gamma_0(N),\chi)+\dim \widetilde{E}_{k+2}(\Gamma_0(N),\chi)>\dim \widetilde{M}_{k+2}(\Gamma_0(N),\chi).
\end{equation*}
By a simple dimension argument, there exists non-zero
\[
f=\sum_{n=0}^\infty {f_nq^n}\in \widetilde{M}_k(\Gamma_0(N),\chi)
\]
such that
\begin{equation*}
f\frac{q\frac{dt}{dq}}{t}\in \widetilde{E}_{k+2}(\Gamma_0(N),\chi).
\end{equation*}
Using our explicit basis for $\widetilde{E}_{k+2}(\Gamma_0(N),\chi)$ from Section \ref{sect_prelim}, we have
\[
f\frac{q\frac{dt}{dq}}{t}=\sum_{n=0}^\infty{c_nq^n}=\sum_i{\frac{\alpha_i}{\beta_i}E_{k+2,\chi_i,\psi_i}(d_iz)}
\]
where $\alpha_i,\beta_i\in\Z$ and $\gcd(\alpha_i,\beta_i)=1$. Recall that $\chi_i$, $\psi_i$ are primitive Dirichlet characters modulo $L_i$,$R_i$. Furthermore, $L_iR_id_i|N$. There exists $D\in\N$ such that
\[
f,\ t,\ f\frac{q\frac{dt}{dq}}{t}\in\Z\left[\tfrac{1}{D}\right]\left[\left[q\right]\right].
\]
By Lemma \ref{lemma_eisensteincongr}, we have $c_{\ell {p^r}}\equiv c_{\ell p^{r-1}} \mod p^r$ for all $\ell, r\in\N$ and for all primes satisfying $p\nmid D\prod_i{\beta_i d_i}$ and $\chi_i(p)=1$ for all $i$. Write $f=\sum {b_nt^n}$.
We have
\begin{equation}\label{theorem1_Eq2}
\left[\sum_{n=0}^\infty {b_nt^n} \right]\frac{dt}{t}=f\frac{dt}{t}=\left[\sum_{n=0}^\infty{c_nq^n}\right]\frac{dq}{q}.
\end{equation}
Observe that $t$ is convergent in some neighborhood of $q=0$ in $\R$. Apply Proposition \ref{JV_Lemma_Extended} with $u=q$ to \eqref{theorem1_Eq2}. We obtain
\[
b_{\ell p^r}\equiv b_{\ell p^{r-1}}\mod p^r
\]
for all $\ell,r\in\N$ and for all primes satisfying $p\nmid D\prod_i{\beta_i d_i}$ and $\chi_i(p)=1$ for all $i$. We obtain the values for $p$ in Table \ref{table_theorem1hypoth} by considering the $\chi_i$ in the explicit basis for $\widetilde{E}_{k+2}(\Gamma_0(N),\chi)$.
\end{proof}

Theorem \ref{theorem1} and Table \ref{table_theorem1hypoth} give us families of congruences in a very general setting, but we can obtain better results in some settings.

\begin{remark}
In Table \ref{table_theorem1hypoth}, we state that congruences are obtained for values of $p$ such that $\left(\frac{p}{3}\right)=1$. However, $\dim \widetilde{M_k}(\Gamma_0(N))=1+k$. The Eisenstein subspace of weight $k$ has a three dimensional subspace spanned by $E_k(z),E_k(3z)$ and $E_k(9z)$, where $E_k(z)$ is the Eisenstein series of weight $k$ on $SL_2(\Z)$. By the argument of Theorem \ref{theorem1}, there exists a modular form of weight $k$ such that the corresponding congruence $b_{\ell p^r}\equiv b_{\ell p^{r-1}}\mod p^r$ holds for all but finitely many primes $p$, not just for primes satisfying $\left(\frac{p}{3}\right)=1$.
\end{remark}

If we assume a slightly stronger hypothesis in \mbox{Theorem \ref{theorem1}}, we obtain \mbox{Corollary \ref{theorem1_cor}}.

\begin{proof}[Proof of Corollary \ref{theorem1_cor}]
Choose an integer $N=2,3,5,7$ or $13$, a real-valued Dirichlet character $\chi$ modulo $N$, and a positive integer $k$ such that
\[
\dim\widetilde{M_k}(\Gamma_0(N),\chi)=\dim\widetilde{M_{k+2}} (\Gamma_0(N),\chi).
\]
By a simple dimension argument, there exists
\[
f_1=\sum_{n=0}^\infty{f_{1,n}q^n}\in\widetilde{M_k}(\Gamma_0(N),\chi)
\]
such that
\[
f_1\frac{q\frac{dt}{dq}}{t}=\sum_{n=0}^\infty{c_{1,n}q^n}=E_{k,\chi_0,\chi}(z).
\]
By \eqref{eisenstein_eq1} and \eqref{eisenstein_eq2} we have $c_{1,\ell p^r}\equiv c_{1,\ell p^{r-1}}\mod p^r$ for all $\ell,r\in\N$ and for all primes $p$. Write $f_1=\sum {b_{1,n}t^n}$. By the same argument as above, we obtain
\[
b_{1,\ell p^r}\equiv b_{1,\ell p^{r-1}}\mod p^r
\]
for all but finitely many primes $p$ and for all $\ell,r \in \N$ .

If $\chi\ne\chi_0$, we can find $f_2\in\widetilde{M_k}(\Gamma_0(N),\chi)$ such that
\[
f_2\frac{q\frac{dt}{dq}}{t}=\sum_{n=0}^\infty{c_{2,n}q^n}=E_{k,\chi,\chi_0}(z).
\]
By \eqref{eisenstein_eq1} and \eqref{eisenstein_eq2} we have $c_{2,\ell p^r}\equiv \chi(p)c_{2,\ell p^{r-1}}\mod p^r$ for all $\ell,r\in \N$ and for all primes $p$. Write $f_2=\sum {b_{2,n}t^n}$. If we inspect the formula for $E_{k,\chi,\chi_0}(z)$ in Section \ref{sect_prelim} we see it, and thus $f_2$, has no constant term since the conductor of $\chi$ is greater than 1. Therefore, we have
\begin{equation}\label{theorem_cor_eq}
\left[\sum_{n=1}^\infty {b_{2,n}t^n} \right]\frac{dt}{t}=f_2\frac{dt}{t}=E_{k,\chi,\chi_0}(z)=\left[\sum_{n=1}^\infty{c_nq^n}\right]\frac{dq}{q}.
\end{equation}
We obtain
\[
b_{2,\ell p^r}\equiv \chi(p) b_{2,\ell p^{r-1}} \mod p^r
\]
by \cite[Proposition 3]{B}. Beuker's result applies here because \eqref{theorem_cor_eq} involves power series instead of Laurent series.
\end{proof}

\section{Examples From Theorem \ref{theorem1}}\label{sect_examples1}
Though Theorem \ref{theorem1} only guarantees congruences for the power series coefficients of $f=\sum {b_nt^n}$ where $b_n\in\Q$, we have many examples where $b_n\in\Z$ for all $n$.
\subsection{Example for $\Gamma_0(2)$}\label{Gamma0(2)_ex}
Let
\[
f=\frac{\eta^{16}(z)}{\eta^8(2z)}\in \widetilde{M}_4(\Gamma_0(2))
\]
and
\[
t=\frac{g}{f}=\frac{\eta^{24}(2z)}{\eta^{24}(z)}.
\]
We have $f\frac{q\frac{dt}{dq}}{t}\in \widetilde{E}_6(\Gamma_0(2))$ with Fourier expansion
\[
f\frac{q\frac{dt}{dq}}{t}=\sum_{n=0}^\infty {c_nq^n}=1+8\sum_{n=1}^\infty {\sigma_5(n)q^n}-512\sum_{n=1}^\infty {\sigma_5(n)q^{2n}}.
\]
Write
\[
f=\sum_{n=0}^\infty {b_nt^n}=1 -16t+ 496t^2 -19456t^3 +860656t^4 -40950016t^5 +2046002176t^6 +\dots.
\]
We have
\begin{equation}\label{Gamma0(2)_eq}
b_{\ell p^r}\equiv b_{\ell p^{r-1}}\mod p^r
\end{equation}
for all primes $p\neq 2$ and all $\ell,r\in\N$ by Theorem \ref{theorem1}. Observing that $c_{\ell \cdot 2^r}\equiv c_{\ell \cdot 2^{r-1}}\mod 2^r$ for all $\ell,r \in \N$, we have \eqref{Gamma0(2)_eq} for $p=2$ as well.

\subsection{Examples for $\Gamma_0(3)$ with character $\left(\frac{\bullet}{3}\right)$}\label{Gamma0(3)_ex}
Let
\[
f_1=\frac{\eta^9(z)}{\eta^3(3z)},\ \ \ \ \ \ f_2=\frac{\eta^9(3z)}{\eta^9(z)},
\]
and
\[
t=\frac{g}{f}=\frac{\eta^{12}(3z)}{\eta^{12}(z)}.
\]
We have $f_1,f_2\in\widetilde{M}_3\left(\Gamma_0(3),\left(\frac{\bullet}{3}\right)\right)$ and $f_1\frac{q\frac{dt}{dq}}{t},f_2\frac{q\frac{dt}{dq}}{t}\in \widetilde{E}_5(\Gamma_0(3),\left(\frac{\bullet}{3}\right))$ with Fourier expansions
\[
f_1\frac{q\frac{dt}{dq}}{t}=\sum_{n=0}^\infty {c_{1,n}q^n}=3E_{5,\chi_0,\left(\frac{\bullet}{3}\right)}(z)
\]
and
\[
f_2\frac{q\frac{dt}{dq}}{t}=\sum_{n=0}^\infty {c_{2,n}q^n}=E_{5,\left(\frac{\bullet}{3}\right),\chi_0}(z)
\].
Write
\[
f_1=\sum_{n=0}^\infty{b_{1,n}t^n}=1-9t+135t^2-2439t^3+48519t^4-1023759t^5+22478121t^6+\dots.
\]
and
\[
f_2=\sum_{n=0}^\infty{b_{2,n}t^n}=t-9t^2+135t^3-2439t^4+48519t^5-1023759t^6+22478121t^7+\dots.
\]
We have
\begin{equation}\label{Gamma0(3)_eq}
b_{1,\ell p^r}\equiv b_{1,\ell p^{r-1}}\mod p^r
\end{equation}
for all primes $p$ satisfying $\left(\frac{p}{3}\right)=1$ and all $\ell,r\in\N$ by Theorem \ref{theorem1}. By \mbox{Corollary \ref{theorem1_cor}}, we actually have \eqref{Gamma0(3)_eq} for all primes $p$.
We have
\begin{equation}\label{Gamma0(3)_eq}
b_{2,\ell p^r}\equiv \left(\frac{p}{3}\right)b_{2,\ell p^{r-1}}\mod p^r
\end{equation}
for all primes $p$ and all $\ell,r\in\N$ by Corollary \ref{theorem1_cor}.

\subsection{Example for $\Gamma_0(4)$ with character $\left(\frac{-4}{\bullet}\right)$}
Let
\[
f=\Theta^2(z)=\left(\sum_{-\infty}^\infty{q^{n^2}} \right)^2,\ \ \ \ \ \ t_1=\frac{\eta^8(z) \eta^{16}(4z)}{\eta^{24}(2z)},\ \ \ \ \ \ t_2=\frac{\eta^{24}(z)\eta^{24}(4z)}{\eta^{48}(2z)}.
\]
We have $f\frac{q\frac{dt_1}{dq}}{t_1}, f\frac{q\frac{dt_2}{dq}}{t_2}\in \widetilde{E}_3(\Gamma_0(4),\left(\frac{-4}{\bullet}\right))$ with Fourier expansions
\[
f\frac{q\frac{dt_1}{dq}}{t_1}=\sum_{n=0}^\infty {c_nq^n}=-4E_{3,\chi_0,\left(\frac{-4}{\bullet}\right)}(z)=1 - 4q - 4q^2 + 32q^3 - 4q^4 - 104q^5 + 32q^6 + 192q^7 - 4q^8 + \dots
\]
and
\[
f\frac{q\frac{dt_2}{dq}}{t_2}=-4E_{3,\chi_0,\left(\frac{-4}{\bullet}\right)}(z)-16E_{3,\left(\frac{-4}{\bullet}\right),\chi_0}(z)=1 - 20q - 68q^2 - 96q^3 - 260q^4 - 520q^5 - 480q^6 +\dots.
\]
Write
\[
f=\sum_{n=0}^\infty{b_nt_1^n}=1+ 4t_1+ 36t_1^2 +400t_1^3+ 4900t_1^4+ 63504t_1^5+ 853776t_1^6+ 11778624t_1^7+\dots
\]
and
\[
f=\sum_{n=0}^\infty{\tilde{b}_nt_2^n}=1+ 4t_2+ 100t_2^2+ 3600t_2^3+ 152100t_2^4+ 7033104t_2^5+ 344622096t_2^6+ \dots.
\]
We have the congruence relations
\begin{equation}\label{Gamma0(4)_eq}
b_{\ell p^r}\equiv b_{\ell p^{r-1}}\mod p^r
\end{equation}
and
\[
\tilde{b}_{\ell p^r}\equiv \tilde{b}_{\ell p^{r-1}}\mod p^r
\]
for $p$ satisfying $\left(\frac{-4}{p}\right)=1$ and for all $\ell,r\in\N$ by Theorem \ref{theorem1}.
M. Somos \cite[A002894, A127776]{Sl} observed that $b_n=\binom{2n}{n}^2$ and $\tilde{b}_n=\left(\frac{2^n}{n!}\prod_{j=0}^{n-1}{4j+1}\right)^2$. Since \mbox{$c_{\ell  p^r}\equiv c_{\ell p^{r-1}}\mod p^r$} for all primes $p$ and for all $\ell,r \in \N$ by \eqref{eisenstein_eq1} and \eqref{eisenstein_eq2}, we actually have \eqref{Gamma0(4)_eq} for all primes $p$.

\subsection{Example for $\Gamma_0(5)$ with character $\left(\frac{\bullet}{5}\right)$: Example 1 Revisited}\label{example1_rev}
Let
\[
f=\frac{\eta^5(z)}{\eta(5z)}\in \widetilde{M}_2\left(\Gamma_0(5),\left(\frac{\bullet}{5}\right)\right)
\]
and
\[
t=\frac{\eta^6(5z)}{\eta^6(z)}.
\]
We have $f\frac{q\frac{dt}{dq}}{t}\in \widetilde{E}_4(\Gamma_0(5),\left(\frac{\bullet}{5}\right))$ with Fourier expansion
\[
f\frac{q\frac{dt}{dq}}{t}=\sum_{n=0}^\infty {c_nq^n}=E_{4,\chi_0,\left(\frac{\bullet}{5}\right)}.
\]
Write
\[
f=\sum_{n=0}^\infty {b_n t^n}=1 -5t+ 35t^2 -275t^3+ 2275t^4 -19255t^5+ 163925t^6 -1385725t^7+ \dots.
\]
We have
\begin{equation}\label{Gamma0(5)_eq1}
b_{\ell p^r}\equiv b_{\ell p^{r-1}}\mod p^r
\end{equation}
for all primes $p$ satisfying $\left(\frac{p}{5}\right)=1$ and for all $\ell,r\in\N$ by Theorem \ref{theorem1}. By Corollary \ref{theorem1_cor} we actually have \eqref{Gamma0(5)_eq1} for all primes $p$.

\subsection{Example for $\Gamma_0(8)$}
Let
\[
f=\frac{\eta(z)^4 \eta(4z)^{10}}{\eta^6(2z)\eta^4(8z)}\in \widetilde{M}_2(\Gamma_0(8))
\]
and
\[
t=\frac{\eta^2(2z)\eta^4(8z)}{\eta^4(z)\eta^2(4z)}.
\]
We have $f\frac{q\frac{dt}{dq}}{t}\in \widetilde{E}_4(\Gamma_0(8))$ with Fourier expansion
\[
f\frac{q\frac{dt}{dq}}{t}=\sum_{n=0}^\infty{c_nq^n}=1-16\sum_{n=1}^\infty {\sigma_3(n)q^{4n}}+256\sum_{n=1}^\infty {\sigma_3(n)q^{8n}}.
\]
Write
\[
f=\sum_{n=0}^\infty {b_n t^n}=1 -4t+ 24t^2 -160t^3+ 1112t^4 -7904t^5+ 57024t^6 -416000t^7+ \dots.
\]
We have
\begin{equation}\label{Gamma0(8)_eq}
b_{\ell p^r}\equiv b_{\ell p^{r-1}}\mod p^r
\end{equation}
for all primes $p\ne 2$ and for all $\ell,r\in\N$ by Theorem \ref{theorem1}. Observing that \mbox{$c_{\ell  \cdot 2^r}\equiv c_{\ell \cdot 2^{r-1}}\mod 2^r$} for all $\ell,r \in \N$, we have \eqref{Gamma0(8)_eq} for $p=2$ as well.

\section{Additional Families of Congruences Via the Method of Theorem \ref{theorem1}}\label{sect_additional_examples}
In Section \ref{sect_examples1}, we computed explicit examples of power series of a modular form $f$ in a modular function $t$. Here, we show how the method from Theorem \ref{theorem1} can be applied to several additional settings.

\subsection{Atkin-Lehner Eigenspaces}
The space $M_k(\Gamma_0(N),\chi)$ can be decomposed into a direct sum of Atkin-Lehner eigenspaces. This can yield additional families of congruences. For example, let $N=10$ and let $M_k(\Gamma_0(10),\epsilon_2,\epsilon_5)$ denote the subspaces of modular forms that are eigenfunctions for the Atkin-Lehner operators $w_2$ and $w_5$ with eigenvalues $\epsilon_2$ and $\epsilon_5$ respectively. We have the dimension formulas,
\begin{center}
\begin{tabular}{c}
$\dim M_k(\Gamma_0(10),+,+)=1-k\slash 2+\lfloor 3k \slash 8\rfloor+2\lfloor k\slash 4\rfloor$\\
$\dim M_k(\Gamma_0(10),+,-)=k\slash 2+\lfloor 3k\slash 8\rfloor-2\lfloor k\slash 4\rfloor$\\
$\dim M_k(\Gamma_0(10),-,+)=k\slash 2-\lfloor 3k\slash 8\rfloor+\lfloor k\slash 4\rfloor$\\
$\dim M_k(\Gamma_0(10),-,-)=k\slash 2-\lfloor 3k\slash 8\rfloor+\lfloor k\slash 4\rfloor$.\\
\end{tabular}
\end{center}
If we set
\[
t(z)=\left(\frac{\eta(2z)\eta(10z)}{\eta(z)\eta(5z)}\right)^4.
\]
Then $h=qdt/tdq\in M_2(\Gamma_0(10),-,+)$. Notice that for $k\equiv 2 \mod 4$ we have
\[
\dim M_k(\Gamma_0(10),+,-)=\dim M_{k+2}(\Gamma_0(10),-,-).
\]
This implies there is a modular form $f$ in $M_{k}(\Gamma_0(10),+,-)$ such that $fh\in E_{k+2}(\Gamma_0(10),-,-)$. Write $f=\sum {b_nt^n}$. By the methods of Theorem \ref{theorem1}, this will give us congruences of the form
\[
b_{\ell p^r}\equiv b_{\ell p^{r-1}}\mod p^r
\]
for all but finitely many primes $p$ subject to some character conditions and for all $\ell,r\in\N$.

\subsection{Hecke Eigenforms}
Hecke eigenforms which are cusp forms have congruence relations similar to those of the Eisenstein series. We can use these relation to obtain congruences similar to those in Theorem \ref{theorem1}. This has been previously studied by \cite{V}. Instead of using the tool proved in Section \ref{sect_congrtool}, we can apply \cite[Proposition 3]{B}. Although Beukers only proves his theorem for differential forms with coefficients in $\Z$ (or $\Z_p$), it is easy to see that a similar statement holds true when we look at differential forms with coefficients in a number field $F$.

For example, choose $N$ a positive integer, a Dirichlet character $\chi$ modulo $N$, and a positive integer $k$ from Table \ref{table_theorem2hypoth}. Therefore, $\dim M_k(\Gamma_0(N),\chi)= \dim M_{k+2}(\Gamma_0(N),\chi)$. Choose a normalized Hecke eigenform
\[
g=\sum_{n=1}^\infty {a_n q^n}\in M_{k+2}(\Gamma_0(N),\chi)
\]
and choose
\[
t=q+\sum_{n=2}^\infty{a_nq^n}\in {M}_0^!(\Gamma_0(N),\psi)
\]
such that the zeros and poles of $t$ are supported at the cusps, where $\psi$ is any Dirichlet character modulo $N$. Therefore, exists $f\in M_k(\Gamma_0(N),\chi)$ such that $f\frac{q\frac{dt}{dq}}{t}=g$.
We know that
\[
a_{\ell p^r}-a_pa_{\ell p^{r-1}}+p^{k+1}a_{\ell p^{r-2}}=0
\]
for all primes $p\nmid N$ and all $r,\ell\in\N$. Write $f=\sum {b_n t^n}$, and we obtain
\[
b_{\ell p^r}-a_p b_{\ell p^{r-1}}+p^{k+1}b_{\ell p^{r-2}}\equiv 0 \mod p^r
\]
for all but finitely many primes $p$ and all $r,\ell\in\N$.

Similarly, let $\sigma_0(N)$ denote the number of divisors of $N$. Suppose
\[
\dim M_{k+2}(\Gamma_0(N))<\dim M_k(\Gamma_0(N))+\sigma_0(N).
\]
Let
\[
g=\sum_{n=1}^\infty {a_n q^n}
\]
 be a normalized Hecke eigenform of weight $k+2$ on $SL_2(\Z)$. By a simple dimension argument, we can find $f\in M_k(\Gamma_0(N))$ such that
\[
f\frac{q\frac{dt}{dq}}{t}=\sum_{d|N} {\alpha_d g(dz)}.
\]
Write $f=\sum {b_nt^n}$. We have
\[
b_{\ell p^r}-a_p b_{\ell p^{r-1}}+p^{k+1}b_{\ell p^{r-2}}\equiv 0 \mod p^r
\]
for all but finitely many primes $p$ and all $r,\ell\in\N$.

\section{A Related Approach To Obtaining Families of Congruences}\label{sect_proof2}
In the proof of Theorem \ref{theorem1}, we use multiplication by $\frac{q\frac{dt}{dq}}{t}$ as a linear transformation to produce infinite families of modular forms $f$ whose power series coefficients in a weakly holomorphic modular function $t$ exhibit congruence relations. We can produce different infinite families of examples by using an alternate linear transformation. In Section \ref{sect_examples2}, we give several examples where the modular function $t$ is not weakly holomorphic.
\begin{definition}
Let $g,f\in M_k(\Gamma_0(N),\epsilon)$. We define

\begin{equation}\label{definition_phi}
\Phi_g(f)=\frac{\theta(f)g-\theta(g)f}{g}.
\end{equation}

\end{definition}

\begin{lemma}\label{proof2_lemma1}
Suppose $g,f\in M_k(\Gamma_0(N),\epsilon)$ have \mbox{$\nu_\infty(g)=\tilde{h}$} and \mbox{$\nu_\infty(f)=h$.} Further suppose that the zeros of $g$ on $\Ha$ are also zeros of $f$. We have \mbox{$\Phi_g(f)\in M_{k+2}(\Gamma_0(N),\epsilon)$} and
\[
\nu_\infty(\Phi_g(f)) = \left\{ \begin{array}{ll}
 h & \mbox{if $h\neq \tilde{h}$};\\
\geq h+1 & \mbox{if $h=\tilde{h}$}.
\end{array} \right.
\]
If $g$ is non-zero on $\Ha$, then
\[
\Phi_g\colon M_k(\Gamma_0(N),\epsilon)\rightarrow M_{k+2}(\Gamma_0(N),\epsilon)
\]
is a linear transformation and $\ker\Phi_g$ is spanned by $g$.
\end{lemma}
\begin{proof}
If the zeros of $g$ are also zeros of $f$ on $\Ha$, then a calculation with the $q$-expansions of $f$ and $g$ gives us the first part of the lemma.

If $g$ is non-zero on $\Ha$ and if $\Phi_g(f)=0$, then $f'g -g'f=0$. Therefore, $\frac{g'}{g}=\frac{f'}{f}$.
\end{proof}

From Table \ref{table_theorem2hypoth}, fix an integer $N$, a character $\chi$ modulo $N$, and an integer $k>0$. Choose any
\[
g=q+\sum_{n=2}^\infty{g_nq^n}\in \widetilde{M}_k(\Gamma_0(N),\chi)
\]
such that the zeros of $g$ are supported at the cusps. In Lemma \ref{proof2_lemma2}, we show there exist $f\in \widetilde{M}_k(\Gamma_0(N),\chi)$ and $t\in \widetilde{\mathcal{M}}_0 (\Gamma_0(N))$ such that \mbox{$\Phi_g(f)=f\frac{q\frac{dt}{dq}}{t}\in \widetilde{E}_{k+2}(\Gamma_0(N),\chi)$.} We then prove Theorem \ref{theorem2} with the same method used for Theorem \ref{theorem1}.

\begin{lemma}\label{proof2_lemma2}
From Table \ref{table_theorem2hypoth}, choose an integer $N$, a character $\chi$ modulo $N$, and an integer $k>0$. Suppose we have $g\in \widetilde{M}_k(\Gamma_0(N),\chi)$ such that its zeros are supported at the cusps and it has Fourier expansion
\[
g=q+\sum_{n=2}^\infty{g_nq^n}.
\]
Then there exists $f\in \widetilde{M}_k(\Gamma_0(N),\chi)$ with Fourier expansion
\[
f=1+\sum_{n=1}^\infty{f_nq^n},
\]
and $\Phi_g(f)\in \widetilde{E}_{k+2}(\Gamma_0(N),\chi)$.
\end{lemma}

\begin{proof}
From Table \ref{table_theorem2hypoth}, choose an integer $N$, a character $\chi$ modulo $N$, and an integer $k>0$. Observe that
\begin{equation}\label{proof2_lemma2_hypoth}
\dim \widetilde{M}_k(\Gamma_0(N),\chi)=\dim \widetilde{M}_{k+2}(\Gamma_0(N),\chi),
\end{equation}
and that
\begin{equation}\label{proof2_lemma2_eq1}
\dim (\widetilde{E}_{k+2}(\Gamma_0(N),\chi))=2
\end{equation}
by dimension formulas \cite[p. 85-92]{DS}. Combining Lemma \ref{proof2_lemma1} and \eqref{proof2_lemma2_hypoth}, we have
\[
\dim(\ker(\Phi_g))+\dim(\operatorname{im}(\Phi_g))=\dim \widetilde{M}_k(\Gamma_0(N),\chi)=\dim \widetilde{M}_{k+2}(\Gamma_0(N),\chi)
\]
and
\begin{equation}\label{proof2_lemma2_eq2}
\dim\left(\operatorname{im}(\Phi_g)\right)=\dim \widetilde{M}_{k+2}(\Gamma_0(N),\chi)-1.
\end{equation}
By \eqref{proof2_lemma2_eq1}, \eqref{proof2_lemma2_eq2}, and basic facts about vector spaces, we have
\[\dim \left(\operatorname{im}(\Phi_g)\cap \widetilde{E}_{k+2}(\Gamma_0(N),\chi)\right)\ge 1.\]
Therefore there exists a non-zero $f\in \widetilde{M}_k(\Gamma_0(N),\chi)$ such that
\[
\Phi_g(f)\in \widetilde{E}_{k+2}(\Gamma_0(N),\chi).
\]
It remains to show that $\nu_\infty(f)=0$.

Suppose by way of contradiction that $\nu_\infty(f)>0$. We have $\nu_\infty(\Phi_g(f))\geq 2$ by Lemma \ref{proof2_lemma1}. However, we have an explicit $\mathbb{C}$-basis for ${E}_{k+2}(\Gamma_0(N),\chi)$ from Section \ref{sect_prelim}. The basis elements described in \mbox{Section \ref{sect_prelim}} have coefficients in $\Q$ for our values of $N$ and $\chi$; thus, these elements form a $\mathbb{Q}$-basis for $\widetilde{E}_{k+2}(\Gamma_0(N),\chi)$.
If $\chi=\chi_0$, then
\[
E_{k+2,\chi_0,\chi_0}(z)=-\frac{B_{k+2,\chi_0}}{2k+4}+q+\dots
\]
and
\[
E_{k+2,\chi_0,\chi_0}(Nz)=-\frac{B_{k+2,\chi_0}}{2k+4}+q^N+\dots
\]
form a basis for $\widetilde{E}_{k+2}(\Gamma_0(N))$. If $\chi\ne \chi_0$, then
\[
E_{k+2,\chi_0,\chi\ }(z)=-\frac{B_{k+2,\chi_0}}{2k+4}+q+\dots
\]
and
\[
E_{k+2,\chi,\chi_0}(z)=q+\dots
\]
form a basis for $\widetilde{E}_{k+2}(\Gamma_0(N),\chi)$. Noting that $\frac{B_{k+2,\chi_0}}{2k+4}\neq 0$ and examining the above bases, we observe that any non-zero element of $\widetilde{E}_{k+2}(\Gamma_0(N),\chi)$ has order of vanishing at most 1 at infinity, which is a contradiction. Hence, there exists
\[
f=1+\sum_{n=1}^\infty{f_n q^n}\in \widetilde{M}_k(\Gamma_0(N),\chi)
\]
such that $\Phi_g(f)\in \widetilde{E}_{k+2}(\Gamma_0(N),\chi)$.
\end{proof}

Observe that by adding multiples of $g$ to $f$ in the previous lemma we can trivially find infinitely many $f\in \widetilde{M}_k(\Gamma_0(N),\chi)$ such that $\Phi_g(f)\in \widetilde{E}_{k+2}(\Gamma_0(N),\chi)$.

\begin{proof}[Proof of Theorem \ref{theorem2}]
From Table \ref{table_theorem2hypoth}, choose an integer $N$, a character $\chi$ modulo $N$, and an integer $k>0$. Choose any
\[
g=q+\sum_{n=2}^\infty{g_nq^n}\in \widetilde{M}_k(\Gamma_0(N),\chi)
\]
such that the zeros of $g$ are supported at the cusps. (For examples of such $g$, choose $A$ and $B$ from Table \ref{table_etaproducts} and an integer $n$ such that $g=A\cdot B^n\in \widetilde{M}_k(\Gamma_0(N),\chi).$)

By Lemma \ref{proof2_lemma2} there exists
\[
f=1+\sum {f_nq^n} \in \widetilde{M}_k(\Gamma_0(N),\chi)
\]
such that
\[
\Phi_g(f)\in \widetilde{E}_{k+2}(\Gamma_0(N),\chi).
\]
Let $t=\frac{g}{f}\in \widetilde{\mathcal{M}}_0 (\Gamma_0(N))$ and observe that $\Phi_g(f)=f\frac{q\frac{dt}{dq}}{t}$. Using our explicit basis for $\widetilde{E}_{k+2}(\Gamma_0(N),\chi)$ from Section \ref{sect_prelim}, we have
\[
\Phi_g(f)=f\frac{q\frac{dt}{dq}}{t}=\sum_{n=0}^\infty{c_nq^n}=\sum_{i=1}^2{\frac{\alpha_i}{\beta_i}E_{k+2,\chi_i,\psi_i}(d_iz)}
\]
where $\alpha_i,\beta_i\in\Z$ and $\gcd(\alpha_i,\beta_i)=1$. Recall that $\chi_i$, $\psi_i$ are primitive Dirichlet characters modulo $L_i$,$R_i$. Furthermore, $L_iR_id_i|N$. In either case, there exists $D\in \N$ such that  we have
\[
t=\sum_{n=1}^\infty {a_n q^n}\in\Z\left[\tfrac{1}{D}\right]\left[\left[q\right]\right],
\]
and
\[
\Phi_g(f)=f \frac{q\frac{dt}{dq}}{t}=\sum_{n=0}^\infty{c_nq^n}\in\Z\left[\tfrac{1}{D}\right]\left[\left[q\right]\right].
\]
By Lemma \ref{lemma_eisensteincongr}, we have $c_{\ell p^r}\equiv c_{\ell p^{r-1}}\mod p^r$ for all $\ell,r\in\N$ and for all primes $p$ satisfying $\chi(p)=1$ and $p\nmid Dd_1d_2\beta_1\beta_2$. Therefore,
\begin{equation}\label{proof2_eq}
\left[\sum_{n=0}^\infty {b_nt^n} \right]\frac{dt}{t}=f\frac{dt}{t}=\left[\sum_{n=0}^\infty{c_nq^n}\right]\frac{dq}{q}.
\end{equation}
Observe that $t$ is convergent in some neighborhood of $q=0$ in $\R$. Apply Proposition \ref{JV_Lemma_Extended} with $u=q$ to \eqref{proof2_eq}. We obtain
\[
b_{\ell p^r}\equiv b_{\ell p^{r-1}}\mod p^r
\]
for all $\ell,r\in\N$ and for all primes satisfying $\chi(p)=1$ and $p\nmid Dd_1d_2\beta_1\beta_2$.
\end{proof}

\section{Examples From Theorem \ref{theorem2}}\label{sect_examples2}
We obtain Example \ref{Gamma0(2)_ex} from Theorem \ref{theorem2} by choosing $g=\frac{\eta^{16}(2z)}{\eta^8(z)}$, $f=\frac{\eta^{16}(z)}{\eta^8(2z)}$, and \mbox{$t=\frac{\eta^{24}(2z)}{\eta^{24}(z)}$.} We obtain Example \ref{Gamma0(3)_ex} from Theorem \ref{theorem2} by choosing $g=\frac{\eta^9(3z)}{\eta^3(z)}$, $f=\frac{\eta^9(z)}{\eta^3(3z)}$, and $t=\frac{\eta^{12}(3z)}{\eta^{12}(z)}$.

\subsection{Example for SL$_2(\Z)$}
Let $g=\Delta=\eta^{24}(z)\in \widetilde{M}_{12}($SL$_2(\Z))$. We calculate that \mbox{$\Phi_g(E_{12})=E_{14}$.} Let $t=\frac{\Delta}{E_{12}}$. Write
\[
E_{12}=\sum_{n=0}^\infty {b_nt^n}=1+ \frac{65520}{691} t+ \frac{98146535760}{477481} t^2+ \frac{27376196366937600}{329939371} t^3+\dots.
\]
We have $b_{\ell p^r}\equiv b_{\ell p^{r-1}}\mod p^r$ for all primes $p\neq 691$ and for all $\ell,r\in\N$ by \mbox{Theorem \ref{theorem2}}.

Similarly, let $\tilde{g}=\Delta^2=\eta^{48}(z)\in \widetilde{M}_{24}($SL$_2(\Z))$. We calculate that
\[
\Phi_{\tilde g}(\Delta\cdot E_{12})=E_{14}\cdot \Delta=\sum_{n=1}^\infty{c_nq^n}.
\]
Let $\tilde{t}=\frac{\Delta^2}{\Delta\cdot E_{12}}=\frac{\Delta}{E_{12}}$. Write
\[
\Delta\cdot E_{12}=\sum_{n=0}^\infty {b_n \tilde{t}^n}=t+\frac{131040}{691}t^2+\frac{200585941920}{477481}t^3+\frac{67613514779865600}{329939371}t^4\dots.
\]
We have $b_{\ell p^r}-c_pb_{\ell p^{r-1}}+p^{25}b_{\ell p^{r-2}}\equiv 0\mod p^r$ for all but finitely many primes $p$ and for all $\ell,r\in\N$ by \cite[Proposition 3]{B}.

\subsection{Example for $\Gamma_0(5)$}
Let
\[
g=\eta^4(z)\eta^4(5z)\in \widetilde{M}_4(\Gamma_0(5))
\]
and
\[
f=1 + 10q + 80q^2 + 260q^3 + 680q^4 + 1390q^5 + 2320q^6 + 3180q^7 +\dots\in \widetilde{M}_4(\Gamma_0(5)).
\]
Further let
\[
t=\frac{g}{f}=q - 14q^2 + 62q^3 + 248q^4 - 4485q^5 + 17012q^6 + 99186q^7 +\dots.
\]
Then $\Phi_g(f)\in \widetilde{E}_6(\Gamma_0(5))$ and
\[
\Phi_g(f)=\sum_{n=0}^\infty {c_nq^n}=\frac{1}{126}E_6(z)+\frac{125}{126}E_6(5z)=1-4\sum_{n=1}^\infty {\sigma_4(n)q^n}-500\sum_{n=0}^\infty {\sigma_4(n)q^{5n}}.
\]
Write
\[
f=\sum_{n=0}^\infty {b_n t^n}=1+10t+220t^2+5800t^3+171400t^4+5428240t^5+180197200t^6\dots.
\]
We have
\begin{equation}\label{Gamma0(5)_eq2}
b_{\ell p^r}\equiv b_{\ell p^{r-1}}\mod p^r
\end{equation}
for all primes $p\nmid 5\cdot 126$ and for all $\ell,r\in\N$ by Theorem \ref{theorem2}. Observing that \mbox{$c_{\ell  p^r}\equiv c_{\ell p^{r-1}}\mod p^r$} for all primes $p$ and for all $\ell,r\in\N$, we actually have \eqref{Gamma0(5)_eq2} for all primes $p$.

\subsection{Example for $\Gamma_0(7)$ with character $\left(\frac{\bullet}{7}\right)$}
Let
\[
g=\eta(z)^3\eta(7z)^3,\ f=1 + 5q + 27q^2 + 56q^3 + 109q^4 + 168q^5 + 280q^6 +\dots \in \widetilde{M}_3\left(\Gamma_0(7),\left(\frac{\bullet}{7}\right)\right),
\]
then
\[
t=\frac{g}{f}=q - 8q^2 + 13q^3 + 100q^4 - 512q^5 - 164q^6 + 8684q^7 +\dots
\]
and $\Phi_g(f)\in \widetilde{E}_5(\Gamma_0(7),\left(\frac{\bullet}{7}\right))$ with Fourier expansion
\[
\Phi_g(f)=\sum_{n=0}^\infty{c_nq^n}=f\frac{q\frac{dt}{dq}}{t}=\frac{1}{16}E_{5,\chi_0,\left(\frac{\bullet}{7}\right)}(z)-\frac{49}{16}E_{5,\left(\frac{\bullet}{7}\right),\chi_0}(z).
\]
Write
\[
f=\sum_{n=0}^\infty{b_nt^n}=1+ 5t+ 67t^2 +1063t^3 +19091t^4+ 368623t^5+ 7475497t^6+ 157030949t^7+\dots.
\]
We have
\[
b_{\ell p^r}\equiv b_{\ell p^{r-1}}\mod p^r
\]
for $p\neq 2$ satisfying $\left(\frac{p}{7}\right)=1$ and for all $\ell,r\in\N$ by Theorem \ref{theorem2}.

\section{Concluding Remarks}
In our theorems, we have only examined modular forms on genus zero $\Gamma_0(N)$. However, many examples exist for modular forms on subgroups of higher genus (see \cite{OS1,V}). An infinite family of examples on a subgroup of higher genus would be interesting.

\section{Acknowledgements}
I would like extend special thanks to Professor Scott Ahlgren for advising me through the research and revisions of this paper. I would also like to thank the University of Illinois - Urbana/Champaign Math Department for providing a supportive environment and valuable resources. I would also like to thank the referee for his helpful comments and additional examples. Finally, all calculations for this paper were done using the computer algebra system SAGE \cite{S}.


\begin{thebibliography}{10}

\bibitem{B}
F. Beukers, Another congruence for the Ap\'ery numbers, \textit{J. Number Th.} \textbf{25} (1987), no. 2, 201--210.

\bibitem{CCS}
H. Chan, S. Cooper and F. Sica, Congruences satisfied by Ap\'ery-like numbers, \textit{Int. J. Number Theory} \textbf{6} (2010), no. 1, 89--97.

\bibitem{CKKO}
H. Chan, A. Kontogeorgis, C. Krattenthaler and R. Osburn, Supercongruences satisfied by coefficients of $_2F_1$ hypergeometric series, \textit{Ann. Sci. Math. Qu\'ebec} \textbf{34} (2010), no. 1, 25--36.

\bibitem{DS}
F. Diamond, J. Shurman, \textit{A first course in modular forms}, Graduate Texts in Mathematics, \textbf{228}, Springer, 2005.

\bibitem{JV}
F. Jarvis, H. Verrill, Supercongruences for the Catalan-Larcombe-French numbers, \textit{Ramanujan J.} \textbf{22} (2010), no. 2, 171--186.

\bibitem{OS1}
R. Osburn, B. Sahu, Congruences via modular forms, \textit{Proc. Amer. Math. Soc.} \textbf{139} (2011), no. 7, 2375--2381.

\bibitem{OS2}
R. Osburn, B. Sahu, Supercongruences for Ap\'ery-like numbers, \textit{Adv. in Appl. Math.} \textbf{47} (2011), no. 3, 631--638.

\bibitem{PS}
C. Peters, J. Stienstra, A pencil of $K3$-surfaces related to Ap\'ery recurrence for $\zeta(3)$ and Fermi surfaces for potential zero, in \textit{Arithmetic of Complex Manifolds (Erlangen 1988)}, Lecture Notes in Mathematics  \textbf{1399} (Springer, 1989), pp. 110--127.

\bibitem{Sl}
N. J. A. Sloane. \textit{The on-line encyclopedia of integer sequences}, published electronically at \mbox{http$\colon\slash\slash$oeis.org}, 2011.
\bibitem{S}
W. Stein et al. Sage Mathematics Software (Version 4.5.3), The Sage Development Team, 2010, \mbox{http$\colon\slash\slash$ www.sagemath.org}.

\bibitem{V}
H. Verrill, Congruences related to modular forms, \textit{Int. J. Number Theory} \textbf{6} (2010), no. 6, 1367--1390.
\end{thebibliography}
\end{document}